\documentclass[11pt]{amsart}
\usepackage[a4paper]{anysize}
\marginsize{2.5cm}{2.5cm}{2.5cm}{2.5cm}
\pdfpagewidth=\paperwidth \pdfpageheight=\paperheight
\usepackage{amsfonts,amssymb,amsthm,amsmath,eucal,tabu}
\usepackage{pgf}
 \usepackage{array}
 \usepackage{pstricks}
 \usepackage{pstricks-add}
 \usepackage{pgf,tikz}
 \usetikzlibrary{automata}
 \usetikzlibrary{arrows}
 \usepackage{indentfirst}
\usepackage{tikz} %Used for drawing picture in LaTeX

\theoremstyle{plain}
\newtheorem{thm}{Theorem}[section]
\newtheorem{theorem}[thm]{Theorem}

\newtheorem{proposition}[thm]{Proposition}

\theoremstyle{definition}
\newtheorem{definition}[thm]{Definition}

\newtheorem{problem}[thm]{Problem}

\newtheorem{thevarthm}[thm]{\varthmname}
\newenvironment{varthm}[1]{\def\varthmname{#1}\begin{thevarthm}}{\end{thevarthm}\def\varthmname{}}
\newenvironment{varthm*}[1]{\trivlist\item[]{\bf #1.}\it}{\endtrivlist}

\renewcommand\ge{\geqslant}
\renewcommand\geq{\geqslant}

\renewcommand\leq{\leqslant}

\newcommand\be{\begin{eqnarray*}}
\newcommand\ee{\end{eqnarray*}}
\newcommand\compact{\itemsep=0cm \parskip=0cm}

\newcommand\Q{\mathbb Q}

\newcommand\newop[2]{\def#1{\mathop{\rm #2}\nolimits}}
\newop\log{log}
\newop\ord{ord}
\newop\Gal{Gal}
\newop\SL{SL}
\newop\Bl{Bl}
\newop\mult{mult}
\newop\mass{mass}
\newop\div{div}
\newop\codim{codim}
\newop\sing{sing}
\newop\vdim{vdim}
\newop\edim{edim}
\newop\Ass{Ass}
\newop\size{size}
\newop\reg{reg}
\newop\satdeg{satdeg}
\newop\supp{supp}
\newop\Neg{Neg}
\newop\Nef{Nef}
\newop\Nefh{Nef_H}
\newop\Eff{Eff}
\newop\Zar{Zar}
\newop\MB{MB}
\newop\MBxC{MB\mathit{(x,C)}}
\newop\NnB{NnB}
\newop\Bigg{Big}
\newop\Effbar{\overline{\Eff}}

\begin{document}

\title{On integral Zariski decompositions of pseudoeffective divisors on algebraic surfaces}

\author{B.\ Harbourne}
\address{Department of Mathematics\\
University of Nebraska\\
Lincoln, NE 68588-0130 USA} \email{bharbour@math.unl.edu}

\author{P. Pokora}
\address{Institute of Mathematics, Pedagogical University\\
Podchor\c a\.zych 2, PL-30-084, Krak\'ow, Poland}
 \email{piotrpkr@gmail.com}

\author{H.\ Tutaj-Gasi\'nska}
%\author{Halszka Tutaj-Gasi\'nska}
\address{
Departament of Mathematics and Computer Sciences, Jagiellonian University\\
{\L}ojasiewicza 6, PL-30-348 Krak\'ow} \email{htutaj@im.uj.edu.pl}

\begin{abstract}
In this note we consider the problem of integrality of Zariski decompositions for pseudoeffective
integral divisors on algebraic surfaces. We show that while sometimes
integrality of Zariski decompositions forces all negative curves to be $(-1)$-curves,
there are examples where this is not true.
\end{abstract}



\date{July 22, 2015}

\thanks{\emph{Acknowledgements}:
The second author would like to thank Roberto Laface for conversations on the topic of this paper. The second author is partially supported by National Science
Centre Poland Grant 2014/15/N/ST1/02102 and the third author is supported by National Science Centre Poland Grant 2014/15/B/ST1/02197.}

\maketitle

%*****************************************************************************
\section{Introduction}
In this note we work over an arbitrary algebraically closed field $K$.
By a negative curve, we mean a reduced irreducible divisor $C$ with $C^2<0$
on a smooth projective surface.
By a $(-k)$-curve, we mean a negative curve $C$ with $C^2=-k<0$.

In \cite{BPS14} the first author with Th.\ Bauer and D.\ Schmitz studied the following problem for algebraic surfaces.

\begin{varthm*}{Question}\rm\
 Let $X$ be a smooth projective surface. Does there exist an integer $d(X) \ge 1$
 such that for every pseudoeffective integral divisor $D$ the denominators in the
 Zariski decomposition of $D$ are bounded from above by $d(X)$?
\end{varthm*}

Such a question is natural when one studies Zariski decompositions \cite{Zar62} of
pseudoeffective divisors since we have the following geometric interpretation.
Given a pseudo-effective integral divisor $D$ on $X$ with
Zariski decomposition $D=P+N$, then for every sufficiently divisible integer $m \ge 1$  we have the equality
   $$
      H^0(X,\mathcal O_X(mD)) = H^0(X, \mathcal O_X(mP))\,,
   $$
which means that all sections of $\mathcal O_X(mD)$ come from the nef line bundle $\mathcal O_X(mP)$.
Sufficiently divisible is required in order to clear denominators in $P$ and obtain Cartier divisors.

If such a bound $d(X)$ exists, then we say that $X$ has \emph{bounded Zariski
denominators}. It is an intriguing question as to whether a given smooth surface
satisfies this boundedness condition.
For example, it was shown in \cite{BPS14} that, somewhat surprisingly, boundedness of Zariski
denominators on a smooth projective surface $X$ is equivalent to $X$ having bounded negativity (i.e.,
that there exists a number $b(X) \in \mathbb{Z}$ such that $C^2 \geq -b(X)$ for every negative
curve $C$).
Bounded negativity has connections to substantial open conjectures.
For example, for a surface $X$ obtained by blowing up $\mathbb{P}^{2}$
at any finite set of generic points, the Segre-Harbourne-Gimigliano-Hirschowitz Conjecture
(i.e., the SHGH Conjecture) \cite{Har86} asserts that $h^1(X, \mathcal O_X(F))=0$ for every effective nef divisor $F$
and in addition that all negative curves on $X$ are $(-1)$-curves.
The Bounded Negativity Conjecture (BNC) is another even older still open conjecture which
asserts that smooth complex projective surfaces all have bounded negativity.
Thus the equivalence of boundedness of Zariski denominators and bounded negativity provides a
new perspective on these conjectures and also sheds some light on links between numerical
information about divisors on a given surface $X$ and the possible
negative curves on $X$.

An interesting criterion for surfaces to have bounded Zariski denominators was given in \cite{BPS14}, as follows,
where we say that a pseudoeffective integral divisor $D$ has an \emph{integral Zariski decomposition} $D = P + N$ if
$P$ and $N$ are defined over the integers (i.e., all coefficients occurring in $P, N$ are integers).

\begin{proposition}
Let $X$ be a smooth projective surface such that for every reduced and irreducible
curve $C$ one has $C^2 \geq -1$. Then all integral pseudoeffective divisors on
$X$ have integral Zariski decompositions.
\end{proposition}

This raises the converse question:

\begin{varthm}{Question}\label{problem1}
Let $X$ be a smooth projective surface having the property that
\begin{center}
$(\ast) \,\,\,\, $ \emph{every integral pseudoeffective divisor $D$ has an integral Zariski decomposition.}
\end{center}
Is every negative curve then a $(-1)$-curve?
\end{varthm}

 The condition $(\ast)$ at first glance seems to be very constraining, so it is
 plausible that Question \ref{problem1} could have an affirmative answer.
 However, by our main result we see that the answer is negative.

\begin{varthm*}{Theorem A}
There exists a smooth complex projective surface $X$ having the property that all
integral pseudoeffective divisors have integral Zariski decompositions yet all
negative curves on $X$ have self-intersection $-2$.
\end{varthm*}

On the other hand, sometimes the answer is affirmative:

\begin{varthm*}{Theorem B}
Let $X$ be a smooth projective surface such that
every integral pseudoeffective divisor $D$ has an integral Zariski decomposition (i.e., $d(X)=1$)
and such that $|\Delta(X)|=1$, where $\Delta(X)$ is the determinant of the intersection form
on the N\'eron-Severi lattice of $X$. Then all negative curves on $X$ are $(-1)$-curves (i.e., $b(X)=1$).
\end{varthm*}

This follows from \cite[Theorem 2.3]{BPS14}, which gives the bound $b(X)\leq d(X)(d(X)!)|\Delta(X)|$.
Thus, for example, if $X$ is a blow up of $\mathbb{P}^{2}$ at a finite set of points, then
$|\Delta(X)|=1$, so if Zariski decompositions are integral on $X$, then $d(X)$ is also 1, hence $b(X)=1$.
Because of the recent interest in blow ups of $\mathbb{P}^{2}$ at finite sets of points
(see, for example, \cite{deFernex, BDHHLPS14, DHNSST15}),
a direct proof in the special case of blow ups of $\mathbb{P}^{2}$ may be useful.
We provide such a proof below.

\section{Results}
Before we present the main result of this note let us recall
the definition of Zariski decompositions.

\begin{definition}[Fujita-Zariski decomposition \cite{Fuj79, Zar62}]
Let $X$ be a smooth projective surface and $D$ a
pseudo-effective integral divisor on $X$. Then $D$ can be
written uniquely as a sum
$$D = P + N$$
of $\Q$-divisors such that
\begin{itemize}\compact
\item[(i)] $P$ is nef,
\item[(ii)] $N$ is effective with negative definite intersection matrix if $N\ne 0$, and
\item[(iii)] $P\cdot C=0$ for every component $C$ of $N$.
\end{itemize}
\end{definition}

Now we are ready to produce the surface whose existence is asserted in Theorem A.

\begin{theorem}
Let $X$ be a smooth projective K3 surface of Picard number $2$ having intersection form
$$
\begin{pmatrix}
-2 & \ 4\\
\ 4 & -2
\end{pmatrix}.
$$
Then all integral pseudoeffective divisors on $X$ have integral Zariski decompositions.
\end{theorem}

\begin{proof}
The existence of such a surface $X$ is a consequence of \cite[(2.9i), (2.11)]{M84} (see also
\cite[Corollary 1.4]{K94}). Let $C_1$ and $C_2$ be the negative curves 
giving a basis for the N\'eron-Severi group with $C_1^2=C_2^2=-2$ and $C_1\cdot C_2=4$. 
Every divisor $D$, up to numerical equivalence, is of the form $mC_1+nC_2$ for integers
$m$ and $n$. If $m<0$ and $n\leq0$ (or vice versa), then clearly $D$ is not effective.
If $m<0$ but $n>0$, then $D\cdot C_2<0$, so $D$ is effective if and only if $D-C_2$ is,
and continuing this way we see that $D$ is effective if and only if $mC_1$ is, but
$m<0$, so $mC_1$ is not effective. Thus $D$ is effective if and only if $m\geq 0$ and $n\geq 0$,
and hence $D$ is effective if and only if it is pseudoeffective.
If $D$ were a negative curve (hence effective) but $D\neq C_1$ and $D\neq C_2$, then
$D\cdot C_1\geq0$ and $D\cdot C_2\geq0$, so $D^2 \geq 0$, contrary to assumption.
Thus $C_1$ and $C_2$ are the only negative curves.
Thus $D$ is nef if and only if $D\cdot C_i\geq 0$ for $i=1,2$;
i.e., if and only if $2m\geq n$ and $2n\geq m$. It is not hard to check that this holds if and only if
$D$ is a nonnegative integer linear combination of $C_1+C_2$, $2C_1+C_2$ and $C_1+2C_2$.

So say $D$ is pseudoeffective and integral; i.e., $D=mC_1+nC_2$ for $m,n\geq0$. By symmetry,
it is enough to assume that $m\geq n$. If $2n\geq m$, then $D$ is nef and the
Zariski decomposition $D=P+N$ of $D$ is integral since $P=D$ and $N=0$.
Now assume $m>2n$. Take $P=n(2C_1+C_2)$ and $N=(m-2n)C_1$.
Then $P$ is nef, $N$ clearly has negative definite intersection matrix and $P\cdot N=0$,
so $D=P+N$ is again an integral Zariski decomposition of $D$.
\end{proof}

Now we provide a proof of Theorem B in the special case mentioned above.

\begin{varthm*}{Theorem B}
Let $\pi:X\to\mathbb{P}^{2}$ be the blow up of a finite set of points $p_1,\ldots,p_s$ (possibly infinitely near).
Suppose that every integral pseudoeffective divisor $D$ has an integral Zariski decomposition.
Then all negative curves on $X$ have self-intersection $-1$ (i.e., are $(-1)$-curves).
\end{varthm*}

\begin{proof}
Denote $\pi^{-1}(p_i)$ by $E_i$ and the total transform of a line by $H$.
We will consider two cases:
in the first case we assume that
none of the points $p_i$ is infinitely near another
(so the points $p_i$ are distinct points of $\mathbb{P}^{2}$)
and in the second case we define: $X_1$ to be the blow up of
$X_0=\mathbb{P}^{2}$ at any point $p_1\in X_0$;
$X_2$ the blow up of $X_1$ at any point $p_2\in X_1$,
so $p_2$ can be infinitely near to $p_1$; etc..
Continuing in this way we eventually have that
$X=X_s$ is the blow up of $X_{s-1}$ at any point $p_s\in X_{s-1}$.
In order to avoid confusion, we indicate the exceptional curve for
the blow up of $p_i\in X_{i-1}$ by $E_{i,i}\subset X_i$, and its total transform
on $X_j$ for $j>i$ by $E_{i,j}$. For simplicity, we denote $E_{i,s}$ by $E_i$.
Thus $E_{i,i}$ is always irreducible, and $E_{i,j}$ is irreducible
if and only if no point $p_l$ for $i<l\leq j$ is infinitely near to $p_i$.

We begin with case 1. Suppose to the contrary that $X$ has a $(-k)$-curve $C$ with
$k > 1$. Since the classes of $H, E_1,\ldots, E_s$ give a basis for the divisor class group of $X$,
up to linear equivalence we can write $C = dH -\sum_{j=1}^{r} b_{i_{j}}E_{i_{j}}$.
Since none of the points $p_i$ is infinitely near another, the only negative curves
with $d=0$ are the $E_i$, and these are $(-1)$-curves, so we must have $d>0$
and $b_{i_j}>0$ with $C^2= d^2 - \sum_{j=1}^{r} b_{i_{j}}^{2} = -k$.

We claim that there is a big integral divisor $D$ such that $D = P + aC$ with $a
\in \mathbb{Q} \setminus \mathbb{Z}$. To see this, note that
if $d'\gg0$ and $a_i>0$ are integers, then
$$A = d'H -\sum_{i=1}^{s} a_{i}E_{i}$$
will be an integral ample divisor.
For $D$ we take $D = A + eC$ for $e \in \mathbb{Z}_{ > 0}$, where we
choose a number $e$ such that $D\cdot C < 0$ so we have
$$0>D\cdot C = (A+eC)\cdot C=A\cdot C-ek=d'd - \sum_{j=1}^ra_{i_j}b_{i_{j}} -ke.$$
Finding the Zariski decomposition of $D$ boils
down to computing $a$. Observe that
$$a = \frac{dd' - \sum_{j=1}^ra_{i_j}b_{i_j} -ke}{-k} = e +\frac{\sum_{j=1}^ra_{i_j}b_{i_{j}} -dd'}{k}.$$
We just need to show that $k$ does not divide $\sum_ja_{i_j}b_{i_{j}} - dd'$.

Suppose that $k$ divides $\sum_ja_{i_j}b_{i_j} - dd'$. Then we replace $A$ by $A+H$ so $d'$ becomes $d'+1$.
If $k$ does not divide $\sum_ja_{i_j}b_{i_{j}} - dd'-d$, then we are done. If $k$ divides this new
number, then it means that $k | d$. In this case we replace $A$ instead by $A+H-E_{i_1}$. Since $H-E_{i_1}$ is nef,
$A+H-E_{i_1}$ is ample, and we get $d'+1$ in place of $d'$ and $a_{i_1}+1$ in place of $a_{i_1}$.
If $k$ does not divide the number $\sum_ja_{i_j}b_{i_j} +b_{i_1} - d(d'+1)$, then we are done.
If $k$ divides this number, then it means that $k|b_{i_1}$. We proceed along the same lines for
all $j$; we are done unless $k|b_{i_j}$ for all $j$. But this is impossible because then $k^2$
would divide $d$ and each $b_{i_j}$, so $k^2$ would divide $k =-(d^{2} -\sum_j b_{i_{j}}^2)$,
where $k$ is an integer bigger than 1.

Now consider case 2. If none of the points is infinitely near another, then we
are in case 1, so we may assume one the points $p_1,\ldots,p_s$
is infinitely near another. Let $p_j$ be the first such point, and let
$p_i$ be the point $p_j$ is infinitely near to. Thus $p_1,\ldots,p_{j-1}$
are points of $\mathbb{P}^{2}$ and $p_j$ is on the exceptional locus $E_{i,j-1}$
of $p_i$ for some $i<j$. After reindexing, we may assume that $i=1$ and $j=2$.
(The only constraint on reindexing is that if $p_v$ is infinitely near to $p_u$, then $v>u$.)
On $X_2$, the curve $E_{1,2}$ has two irreducible components, $E_{1,2}=E+E_{2,2}$.
Thus we have $E^2=-2$.

Here we take $D=3H-E_{1,2}-2E_{2,2}$. Up to linear equivalence, we can write
$2D=6H-2E_{1,2}-4E_{2,2}=6(H-E_{1,2})+4E=[6(H-E_{1,2})+3E]+E$ so $D$ is pseudoeffective (in fact it is linearly
equivalent to an effective divisor).
Since $H-E_{1,2}$ is nef and $(6(H-E_{1,2})+3E)\cdot E=0$, we see that
$D=P+N$ for $P=(6(H-E_{1,2})+3E)/2$ and $N=E/2$ is a nonintegral Zariski decomposition on $X_2$.

However, $\pi$ factors as $\pi: X\stackrel{\pi_2}{\rightarrow} X_{2} \rightarrow\mathbb{P}^{2}$,
where $\pi_2$ is the sequential blow up of the points $p_3,\ldots, p_s$.
Take the pull-back
$$2\pi_2^*(D)=\pi_2^*(2D) = \pi_2^*(2P + 2N) = \pi_2^*(2P) + \pi_2^*(2N)$$
and note that this is a Zariski decomposition.
Indeed, observe that the pull-back of a nef divisor
is nef, so $\pi_2^{*}(2P)$ is nef, and
$\pi_2^{*}(2P)\cdot\pi_2^{*}(2N) = 4P\cdot N = 0$. The only things left to
show is that the intersection matrix of $\pi_2^{*}(2N)$ is negative definite and that
$\pi_2^{*}(2N)/2$ is not integral.
However, up to linear equivalence, $N=E=E_1-E_2$, and the total transforms of
$E_1$ and $E_2$ under $\pi_2$ are in the span (in the divisor class group)
of $E_1,\ldots, E_s$. (Just as the components of the total transform of $E_1$ under
$X_2\to X_1$ are $E_2$ and $E=E_1-E_2$, every component of $E_i$ on $X$
is a linear integer combination of $E_i,\ldots,E_s$.)

Since $\pi_2^{*}(2N)$ is in the span of $E_i,\ldots,E_s$ and this span is negative definite,
the intersection matrix for the components of $\pi_2^{*}(2N)$ is also negative definite.
Since $((\pi_2^{*}(2N))/2)^2=N^2=(E/2)^2=-1/2$, we see that $\pi_2^{*}(2N))/2$ is not integral.
Thus $\pi_2^*(D)=\pi_2^*(2P)/2 + \pi_2^*(2N)/2$ gives a nonintegral Zariski decomposition on $X$.
\end{proof}

We end by posing the following problem.

\begin{problem}
Classify all algebraic surfaces with $d(X)=1$.
\end{problem}

%*****************************************************************************

%***************************************************************************** % Addresses

%*****************************************************************************

\end{document}